[1994/06/01]

\documentclass[12pt,reqno,intlimits,a4paper,english]{amsart}

\usepackage{pstricks,graphicx}
\usepackage{calc}
\usepackage{babel}
\usepackage{amssymb}
\usepackage{environ}

\numberwithin{equation}{section}
\newcounter{hours}\newcounter{minutes}

\newlength{\Indent}
\newlength{\Parskip}
\theoremstyle{plain}
\newtheorem{thm}{Theorem}[section]     
\newtheorem{prop}[thm]{Proposition}
\newtheorem{coroll}[thm]{Corollary}

\makeatletter
\newcommand{\probleminput}[1]{\gdef\@probleminput{#1}}
\NewEnviron{problem}{
\probleminput{}
  \BODY
  \par\addvspace{.5\baselineskip}
  \noindent
   \textsc{Problem 1}\\  \@probleminput  \hfill $\square$\\
  \par\addvspace{.5\baselineskip}
}

\theoremstyle{remark}

\newtheorem{Remark}[thm]{Remark}
\newenvironment{remark}{\begin{Remark}}{\qed\end{Remark}}

\theoremstyle{definition}
\newtheorem{Defin}[thm]{Definition}
\newenvironment{defin}{\begin{Defin}}{\qed\end{Defin}}

\DeclareMathOperator{\Div}{div}

\newcommand{\RN}{\R^{N}}
\newcommand{\R}{\mathbb{R}}
\newcommand{\N}{\mathbb{N}}
\newcommand{\abs}[1]{\lvert#1\rvert}

\newcommand{\di}{\,\text{\rmfamily\upshape d}}

\newcommand{\pder}[2]{\frac{\partial #1} {\partial #2}}

\newcommand{\Om}{\varOmega}
\newcommand{\eps}{\varepsilon}

\newcommand{\CC}{\mathcal{C}}

\def\X{{\mathcal X}}

\newcommand{\lfint}{\lambda_{\textup{int}}}
\newcommand{\lfout}{\lambda_{\textup{out}}}
\newcommand{\lfboth}{\lambda}
\newcommand{\lfbothe}{\lambda}

\newcommand{\Omint}{\Om_{\textup{int}}}
\newcommand{\Omout}{\Om_{\textup{out}}}
\newcommand{\Memb}{\varGamma}

\newcommand{\Permemb}{\varGamma}

\newcommand{\XX}{\X}
\newcommand{\Hzero}{\mathcal{H}_0(\Memb)}
\newcommand{\Helle}{\mathcal{H}_\ell(\Memb)}
\newcommand{\delle}{d_{\mathcal{H}_\ell}}
\newcommand{\beltrami}{\Delta^{\!\!B}}
\newcommand{\beltramigrad}{\nabla^{B}}
\newcommand{\beltramidiv}{\Div^{B}}

\newcommand{\spaziosoleps}{L^2\big(0,T;\XX_0(\Om)\big)}
\newcommand{\wto}{\rightharpoonup}

\newcommand{\const}{\gamma}

\newcommand{\nOmout}{\Om^{\text{bulk}}}

\newcommand{\nOmoutuno}{\Om^{\text{int}}}
\newcommand{\nOmoutdue}{\Om^{\text{out}}}
\newcommand{\nOmint}{\Om^{\text{m}}}

\newcommand{\nMemb}{\partial\nOmint}

\newcommand{\soluz}{u}
\newcommand{\soluzdelta}{u^{\delta}}
\newcommand{\soluzdeltat}{u^{\delta}_{t}}
\newcommand{\soluzt}{u_{t}}
\newcommand{\ndfbothe}{A}
\newcommand{\nlfbothe}{B}
\newcommand{\nlfbothed}{B^\delta}
\newcommand{\testeta}{\phi^{\eta}}

\makeindex
\begin{document}

\title
{Well-posedness of two pseudo-parabolic problems for electrical conduction in heterogeneous media}
\author{M. Amar$^\dag$ -- D. Andreucci$^\dag$ -- R. Gianni$^\ddag$ -- C. Timofte$^\S$\\
\hfill \\
$^\dag$Dipartimento di Scienze di Base e Applicate per l'Ingegneria\\
Sapienza - Universit\`a di Roma\\
Via A. Scarpa 16, 00161 Roma, Italy
\\ \\
$^\ddag$Dipartimento di Matematica ed Informatica\\
Universit\`{a} di Firenze\\
Via Santa Marta 3, 50139 Firenze, Italy
\\ \\
$^\S$University of Bucharest\\
Faculty of Physics\\
P.O. Box MG-11, Bucharest, Romania
}

\begin{abstract}
We prove a well-posedness result for two pseudo-parabolic problems,
which can be seen as two models for the same electrical conduction phenomenon in
heterogeneous media, neglecting the magnetic field.
One of the problems is the concentration limit of the other one, when the thickness of
the dielectric inclusions goes to zero.
The concentrated problem involves
a transmission condition through interfaces, which is mediated by a suitable Laplace-Beltrami type equation.
\medskip

  \textsc{Keywords:} Existence and uniqueness, Laplace-Beltrami operator, interfaces, pseudo-parabolic
  equations.

  \textsc{AMS-MSC:} 35A01, 35K70, 58J05, 35M33
  \bigskip

  \textbf{Acknowledgments}: The first author is member of the \emph{Gruppo Nazionale per l'Analisi Matematica, la Probabilit\`{a} e le loro Applicazioni} (GNAMPA) of the \emph{Isti\-tuto Nazionale di Alta Matematica} (INdAM).
The second author is member of the \emph{Gruppo Nazionale per la Fisica Matematica} (GNFM) of the \emph{Istituto Nazionale di Alta Matematica} (INdAM).
The last author wishes to thank \emph{Dipartimento di Scienze di Base e Applicate per l'Ingegneria} for the warm hospitality and \emph{Universit\`{a} ``La Sapienza" of Rome} for the financial support.

\end{abstract}

\maketitle

\ifx\Versione\UnDeFiNeD\else
\begin{center}
\Versione
\end{center}
\fi


\section{Introduction}\label{s:introduction}
Composite materials are experiencing an increasing popularity in
different fields of science and applications. In particular, they have
a relevant role in the study of electrical conduction. A typical
geometry is the one of an electric conductor in which inclusions of a
possibly different conductive material are inserted, separated from
the surrounding by a dielectric layer, also called membrane. The
latter is in principle thick, i.e. $N$-dimensional, but for reasons of
simplicity it is often replaced with a thin (i.e. $(N-1)$-dimensional)
interface, also in view of the homogenization limit.

It is therefore necessary to investigate the behavior of the composite
medium when the thickness $\eta$ of the insulator vanishes, thus reducing
the membrane to a surface. Obviously, a suitable scaling of the
relevant physical quantities is required in this process. Some of the
authors in \cite{Amar:Andreucci:Bisegna:Gianni:2006a} performed the limit $\eta\to0$, assuming
essentially that the relative dielectric constant $\eps_r$ scales in
such a way that it tends to zero as the thickness of the membrane
$\eta$. In this case, it has been proved in \cite{Amar:Andreucci:Bisegna:Gianni:2006a} that in the
limiting problem the current is continuous across the interfaces and
the time derivative of the jump of the potential across the interfaces
is proportional to the current. This setting is suitable for
electrical conduction in biological tissues of alternating currents at
the radiofrequency range, assuming the magnetic field to be
negligible (see \cite{Amar:Andreucci:Bisegna:Gianni:2003a,Amar:Andreucci:Bisegna:Gianni:2004a,Amar:Andreucci:Bisegna:Gianni:2010,
Amar:Andreucci:Bisegna:Gianni:2013} and the references therein).

It is clear that other scalings are possible. We investigate here a
problem arising from the somewhat more traditional scaling of
concentration of capacity, where in practice $\eps_{r}$ scales as
$1/\eta$ in the limit. This fact produces, in the concentration
limit, the appearance of a transmission condition across the resulting
$(N-1)$-dimensional membrane ruled by the Laplace-Beltrami operator.
Similar problems in the context of thermal diffusion are studied, for instance,
in \cite{Savare:Visintin:1997}, where the authors consider a concentrated
capacity model in the framework of abstract evolution equations, and
in \cite{Amar:Gianni:2016C,Amar:Gianni:2016B,Amar:Gianni:2016A}, where
the concentration and the homogenization of a parabolic system involving
the tangential derivatives on the interface are investigated.
For a general survey on tangential operators we refer to \cite{Hebey:1996,Rosenberg:1997} and the references
therein, as well as for the well-posedness
in concentrated capacity problems arising in the framework of the heat conduction
we may recall, for instance,  \cite{Andreucci:1990,Magenes:1998} and the references therein.
Nevertheless, to the best of our knowledge, this problem has not
yet been considered in the framework of electric conduction described
above. It seems, however, that it may be physically relevant when
dealing with new dielectric materials made available by recent
technology (see \cite{AhAiAh2016}),
where the dielectric constant is very
high. The details and motivations of the concentration limit are
worked out in a forthcoming paper \cite{Amar:Andreucci:Gianni:Timofte:2017A}.
However, for the reader's convenience, we provide here a brief formal derivation of the
thin membrane problem starting from the thick one, in order to enlighten
the connection between the two pseudo-parabolic problems addressed in this paper,
on which we focus our main mathematical interest.

More precisely, the thin membrane problem consists of two elliptic
equations set in the exterior and interior conductive phases. Here, the
dependence on time is merely parametric. The two equations are
coupled at the interfaces, where the potential is continuous. In
addition, on these surfaces the time derivative of the potential solves
an equation for the Laplace-Beltrami operator. The jump of the current
across the interface acts as a source in this pseudo-parabolic
problem. Let us remark that such an equation is due to the
dielectric nature of the membranes. This mathematical problem appears
to be rather non-standard and new in the literature. Indeed, the elliptic and pseudo-parabolic
character of the differential equations imposes as a compatibility
condition the fact that the global flux of the potential at each one
of the interfaces must be essentially assigned.
Problems with a total flux boundary condition
can be encountered also in other contexts (see, e.g.,
\cite{Cioranescu:Damlamian:Li:2013,Li:1989,Li:Zheng:Tan:Shen:1998} and the
references quoted therein). The novelty
in our case is its coupling with an evolutive problem and the fact
that the solution on the interfaces is not required to be
constant. Moreover, the dependence on time and therefore the presence of initial
data raise some nontrivial issues.

The necessary presence at the same time of the total flux condition, the evolutive pseudo-parabolic character,
and the Laplace-Beltrami operator creates technical problems, starting even from the choice of suitable functional spaces,
for which we have been unable to find any specific reference in the literature.

Even the mathematical model for thick membranes shares the interesting
features of the concentrated scheme. In addition, it exhibits the
pseudo-parabolic equation in the bulk of the domain, that is in the
dielectric phase, which makes the degeneracy possibly stronger.
However, also for this problem, to the best of our knowledge, there are no results
of well-posedness in the existing literature.

We recall that similar problems involving pseudo-parabolic equations can appear in several different
contexts (see \cite{Cances:Choquet:Fan:Pop:2010,Cao:Pop:2015,Fan:Pop:2011,Cuesta:2000,Mikelic:2010,Ptashnyk:2007,Ptashnyk_bis:2007}).
However, in these papers the pseudo-parabolic part is always coercive and there are no interfaces; moreover, a pure parabolic term appears.
This kind of problems can model, for instance, fluid flow in porous media or
heat conduction in two-temperature systems and
are also used to regularize ill-posed transport problems
(see, also, \cite{Barenblatt:1990,Barenblatt:1993,Dull:2006,Vromans:2017} and the references therein).

We present here a proof of the well-posedness of both problems introduced above,
relying on two different approaches. In both cases, we see that the
compatibility condition mentioned above and, independently, the
parametric dependence on time in the conductive phases prevent the
solution to attain the prescribed initial data other than in the sense of
electrical currents, which is actually the correct one from a physical
viewpoint (see Remarks \ref{r:r6bis} and \ref{r:r8}). The problem with
a $N$-dimensional dielectric phase is treated by approximation,
introducing a suitable sequence of coercive problems. Our approach to
the problem with $(N-1)$-dimensional interfaces is based on a fixed point
argument, and must be carefully tuned in such a way that our
contractive operator preserves the relevant compatibility
conditions. This is made possible by Proposition~\ref{t:a3} which,
though related to results in \cite{Li:Zheng:Tan:Shen:1998}, seems to be new in its present
formulation, and allows some regularity estimates (see Remark \ref{r:r4}).

The paper is organized as follows: in Section~\ref{s:threeD_problem}, we present our main results and the geometrical setting for both problems,
together with a formal motivation of the concentration procedure.
In Section~\ref{s:exist_thick}, we give proofs for the problem with thick membranes and, in Section~\ref{s:exist_micro}, for the one with thin interfaces.

\section{Problems and Main Results}\label{s:threeD_problem}

\subsection{Tangential derivatives}
\label{s:LB}
Let $\phi$ be a ${\mathcal C}^2$-function,  $\mathbf{\Phi}$ be a
${\mathcal C}^2$-vector function and $S$ a smooth surface in $\R^N$
with normal unit vector $n$.
We recall that the tangential gradient of $\phi$ is given by
\begin{equation}\label{eq:a5biss}
\beltramigrad\phi=\nabla\phi-(n\cdot\nabla\phi)n
\end{equation}
and the tangential divergence of $\Phi$ is given by
\begin{multline}\label{eq:a3bis}
\beltramidiv\mathbf{\Phi}
=\Div\mathbf{\Phi}- (n\cdot\nabla\mathbf{\Phi}_i)n_i-(\Div n)(n\cdot\mathbf{\Phi})
\\
=\beltramidiv \left(\mathbf{\Phi}- (n\cdot\mathbf{\Phi})n\right)
=\Div\left(\mathbf{\Phi}- (n\cdot\mathbf{\Phi})n\right)\,,
\end{multline}
where, taking into account the smoothness of $S$,
the normal vector $n$ can be naturally defined in a small neighborhood of $S$ as
$\frac{\nabla d}{|\nabla d|}$, where $d$ is the signed distance from $S$.
Moreover, we define the Laplace-Beltrami operator as
\begin{equation}
\label{eq:beltrami}
\beltrami\phi =\Div^B(\beltramigrad\phi)\,.
\end{equation}
Finally, we recall that on a regular surface $S$ with no boundary (i.e. when $\partial S=\emptyset$)
we have
\begin{equation}\label{eq:a66}
\int_S \beltramidiv \mathbf{\Phi}\di\sigma =0 \,,
\end{equation}
for any ${\mathcal C}^2$-vector function $\mathbf{\Phi}$.

Here and in the following, the operators $\Div$ and $\nabla$, as well as $\Div^B$ and $\beltramigrad$, act
only with respect to the space variable $x$.
\medskip

\subsection{The problem with thick membranes}
\label{s:thick}

We consider first the problem where the insulating membranes have  a positive thickness, as displayed in Fig. \ref{fig:thick}.
\begin{figure}[htbp]
  \begin{center}
    \begin{pspicture}(12,4)
  \rput(6,2){
    \psccurve[linewidth=1pt](-2,-1)(-1,0)(0,-1)(1,-1)(2,0)(3,1)(4,1.5)(0,2)(-3,2)
    \rput[r](-3.5,0){\mbox{$\Om$}}
      \psellipse[linewidth=1pt,fillstyle=solid,fillcolor=gray](-.5,1.3)(0.45,0.6)
      \psellipse[linewidth=1pt,fillstyle=solid,fillcolor=gray](2,1)(0.3,0.5)
      \psellipse[linewidth=1pt,fillstyle=solid,fillcolor=gray](-1.5,.8)(0.4,0.4)
      \psellipse[linewidth=1pt,fillstyle=solid,fillcolor=gray](-2.5,.5)(0.4,0.5)
      \psellipse[linewidth=1pt,fillstyle=solid,fillcolor=gray](1,0)(.5,.4)%
      \psellipse[linewidth=1pt,fillstyle=solid,fillcolor=white](-.5,1.3)(0.3,0.1)
      \psellipse[linewidth=1pt,fillstyle=solid,fillcolor=white](2,1)(0.1,0.1)
      \psellipse[linewidth=1pt,fillstyle=solid,fillcolor=white](-1.5,.8)(0.2,0.3)
      \psellipse[linewidth=1pt,fillstyle=solid,fillcolor=white](-2.5,.5)(0.1,0.15)
      \psellipse[linewidth=1pt,fillstyle=solid,fillcolor=white](1,0)(.3,.2)
      \rput[r](1.6,1){\mbox{$\nOmint_{i}$}}%
  }
\end{pspicture}
 
    \caption{The geometrical setting in the case of thick membranes: here $\nOmint$ is the gray region, while $\nOmout$ is the white region.
    The connected components of $\nOmint$ are labelled as $\nOmint_{i}$.}
    \label{fig:thick}
  \end{center}
\end{figure}
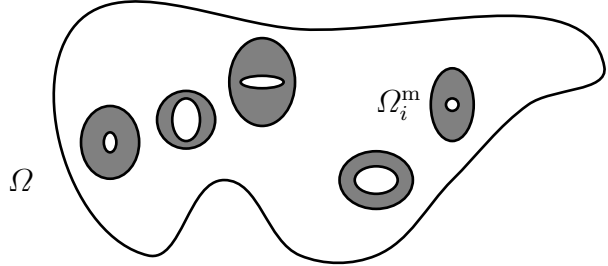
To this purpose, let $\Om$ be a smooth open connected bounded subset of $\RN$
and let us write $\Om$
as $\Om=\nOmout\cup\nOmint\cup\nMemb$, where $\nOmout$ and $\nOmint$
are two disjoint open subsets of $\Om$. More precisely, $\nOmint$ is the union of the isolating membranes
with boundary $\nMemb$, while we assume that $\nOmout={\nOmoutuno}\cup{\nOmoutdue}$, where
$\nOmoutdue$, $\nOmoutuno$ correspond to the two conductive regions separated by
the membrane $\nOmint$ and we denote $\nMemb =
(\partial{\nOmoutuno}\cup\partial{\nOmoutdue})\cap \Om$.
We assume that $\nOmoutdue$ is connected and that the boundary $\partial \nOmint$ is smooth.
Finally, given $T>0$, we denote by $\Om_{T}=\Om\times(0,T)$.
More in general, for any spatial domain $G$, we denote $G_{T}=G\times(0,T)$.

Let $\lfint$, $\lfout$, $\alpha$ be strictly positive constants and set
$\ndfbothe (x) = \lfint$ in $\nOmoutuno$, $\ndfbothe (x) = \lfout$ in
$\nOmoutdue$, $\ndfbothe (x) = 0$ in $\nOmint$, $\nlfbothe (x) = 0$ in
$\nOmoutuno\cup\nOmoutdue$, $\nlfbothe (x) = {\alpha}$ in $\nOmint$.

For a given function $\overline u_0\in H^{1}(\nOmint)$, we
will denote by $\widetilde u_0$ an extension to the whole of $\Om$ of
$\overline u_0$, such that $\widetilde u_0\in H^1_0(\Om)$ and
$\Vert \widetilde u_0\Vert^{}_{H^1_0(\Om)}\leq \gamma \Vert \overline
u_0\Vert^{}_{H^{1}(\nOmint)}$.
Notice that we will use the $H^1_0(\Om)$ regularity of the initial data
in the proof of Theorem \ref{t:a1}, as required by the relaxation technique applied there.
Finally, let $f\in L^2(\Om_T)$.

We consider the problem for $\soluz$
given by
\begin{alignat}2
  \label{eq1}
  -\Div(\ndfbothe \nabla\soluz+\nlfbothe \nabla \soluzt)&=f\,,&\qquad &\text{in $\Om_T$;}
  \\
  \label{eq5}
 \nabla\soluz(x,0)&=\nabla\overline u_0(x)\,,&\qquad&\text{in $\nOmint$.}
\end{alignat}

\begin{defin}
  \label{d:thick_sol}
  We say that $\soluz(x,t)\in L^2\big(0,T;H^1_0(\Om)\big)$ is a weak solution to problem \eqref{eq1}--\eqref{eq5} if
  \begin{equation}\label{eq:a14}
    \int_{0}^{T}\!\!\int_{\Om} \ndfbothe \nabla\soluz\cdot\nabla{\phi}\di x\di t-
    \int_{0}^{T}\!\!\int_{\Om} \nlfbothe \nabla \soluz\cdot\nabla\phi_t \di x\di t=
    \int_{0}^{T}\!\!\int_{\Om}f \phi \di x\di t+
    \int_{\Om} \nlfbothe \nabla\overline u_0\cdot\nabla{\phi}(0)\di x\,,
  \end{equation}
  for every test function $\phi\in \CC^\infty(\overline{\Om_T})$ such that $\phi$ has compact support in $\Om$ for every
  $t\in(0,T)$ and $\phi(\cdot,T)=0$ in $\Om$.
\end{defin}

First, we note that, by formally testing \eqref{eq1} and \eqref{eq5} with $u$, integrating in time and using Poincar\'{e}'s inequality and
Gronwall's lemma, one can derive the energy inequality
\begin{equation}\label{energy1}
\int_{0}^{T}\int_{\nOmoutdue\cup\nOmoutuno} \abs{\nabla \soluz}^{2}
  \di x\di\tau
  +  \sup_{t\in(0,T)}\int_{\nOmint} \abs{\nabla \soluz}^{2}(t) \di x
\le \gamma(\Vert f\Vert^{2}_{L^2(\Om)}+
\Vert{\nabla \overline u_0}\Vert^{2}_{L^2(\nOmint)})\,,
\end{equation}
where $\gamma=\gamma(\lfint,\lfout,\alpha,\nOmint)$.  In the next
proposition, we make this remark rigorous; this point is not obvious
due to the characteristic feature of the problems at hands, that is
the fact that they do not necessarily preserve the value of the
initial data (which actually may be locally changed by a constant).

\begin{prop}\label{r:r6}
Let $u\in L^2\big(0,T;H^1_0(\Om)\big)$ be a solution of problem \eqref{eq1}--\eqref{eq5}.
Then, $u$ satisfies the energy estimate \eqref{energy1}.
\end{prop}

\begin{proof}
Reasoning as in \cite[p.158 ff.]{LSU}, we can prove that $\nabla u\in \CC^0\big(0,T;L^2(\nOmint)\big)$. Hence, by Gronwall's lemma
and Poincar\'{e}'s inequality, we obtain that
\begin{equation}\label{energy11}
\int_{0}^{T}\int_{\nOmoutdue\cup\nOmoutuno} \abs{\nabla \soluz}^{2}
  \di x\di\tau
  +  \sup_{t\in(0,T)}\int_{\nOmint} \abs{\nabla \soluz}^{2}(t) \di x
\le \gamma(\Vert f\Vert^{2}_{L^2(\Om)}+
\Vert{\nabla u(0)}\Vert^{2}_{L^2(\nOmint)})\,.
\end{equation}
Moreover, from the weak formulation
\eqref{eq:a14}, it follows that
\begin{equation}\label{eq:a1}
\int_{\Om} B(\nabla u(0)-\nabla \overline u_0)\cdot\nabla \varphi\di x=0\,,
\end{equation}
for all $\varphi\in\CC^1(\overline\Om)$. By the asserted continuity of $\nabla u(t)$, we infer that, up to modifying $u(t)$ by a constant
$k_i(t)$ in each connected component $\nOmint_i$ of $\nOmint$, we may assume that, as $t\to 0$, $u(t)\to v_i$ in $L^2(\nOmint_i)$. For example, 
we may enforce the condition that $u(t)$ has zero average on each component $\nOmint_{i}$.
Therefore, recalling that $B=0$ in $\nOmout$ and setting $g=v_i-\overline u_0$ in $\nOmint_i$, from \eqref{eq:a1} it follows
that in each connected component $\nOmint_{i}$ we have
\begin{equation*}
\int_{\nOmint_{i}} B\nabla g\cdot\nabla \varphi\di x=0\,,
\end{equation*}
for all $\varphi\in\CC^1(\overline{\nOmint_{i}})$, which is exactly the weak formulation of the Neumann problem
\begin{equation*}
\begin{aligned}
-\Div (B\nabla g)& =0\,, &\quad & \text{in $\nOmint_i$;}
\\
\pder{g}{\nu}&=0\,, &\quad & \text{on $\nMemb_i$.}
\end{aligned}
\end{equation*}
This implies that $v_i=\overline u_0+c_i$, for $c_i\in\R$. Hence, $\nabla u(x,0)=\nabla \overline u_0(x)$
a.e. in $\nOmint$ and the thesis is proven.
\end{proof}

Clearly, estimate \eqref{energy1} and the linearity of problem
\eqref{eq1}--\eqref{eq5} ensure that, if a solution does exist in
$L^2\big(0,T;H^1_0(\Om)\big)$, it is unique and depends continuously on $\nabla\overline u_0$.  Hence, we have just to
prove existence. This will be done in Section~\ref{s:exist_thick}, proving the
following general result.

\begin{thm}\label{t:a2}
Let $\ndfbothe,\nlfbothe, f $ and $\overline u_0$
be as above. Then, for any given $T>0$, problem \eqref{eq1}--\eqref{eq5}
admits a unique solution $\soluz\in L^2\big(0,T;H^1_0(\Om)\big)$.
\end{thm}

\begin{remark}\label{r:r6bis}
 Assume, for the moment, that the source $f$ is sufficiently regular to ensure that $\nabla u$ has a trace for $t=0$ as an $L^2$-function in the whole of $\Om$.

Notice that, from \eqref{eq1}--\eqref{eq:a14}, it follows that the solution $u$ satisfies, for each connected component
$\nOmint_i$ of $\nOmint$,
\begin{equation}\label{eq:a6n}
\int_{\partial\nOmoutdue\cap\nMemb_i}\lfout\pder{u}{\nu}\di \sigma =\int_{\nOmoutuno_i\cup\nOmint_i}f\di x\,,
\end{equation}
in a weak sense (see Remark \ref{r:r20}).
In general, this condition at time $0$
is not automatically satisfied for every choice of the initial datum $\overline u_0$.
However, reasoning as in Proposition \ref{t:a3} with $\Memb_i=\partial\nOmoutdue\cap\nMemb_i$, it is possible to
prove that $\overline u_0$ can be modified by a suitable constant $c_i$ in every $\nOmint_i$ in such a way that \eqref{eq:a6n} is fulfilled.
Clearly, this does not affect the initial condition for $\nabla u$, in accordance with the
fact that problem \eqref{eq1}--\eqref{eq:a14} requires an initial condition only for the gradient of the solution and not
for the solution itself.
Therefore, the solution does not assume exactly the initial condition $\overline u_0$; roughly speaking, it ``rearranges by itself"
the prescribed initial condition by adding to $\overline u_0$ in each connected component of $\nOmint$ the previously quoted
constant $c_i$, in order for \eqref{eq:a6n} to hold true.

However, we stress again that without suitable assumptions on the source $f$  with respect to the time dependence
(for instance $f\in H^1\big(0,T;L^2(\Om)\big)$), it is not possible to guarantee
that $u$ has a trace for $t=0$ in the whole of $\Om$,
since in $\nOmout\cup\nOmint$ the problem \eqref{eq1} displays only a parametric dependence on $t$.

\end{remark}

\subsection{Formal concentration}
\label{ss:concentration}
We devote this subsection to formally justify the relationship between the problem with thick membranes
and the one with thin membranes. The rigorous proof of this result can be found in \cite{Amar:Andreucci:Gianni:Timofte:2017A}.

 Assume that $\nOmint$ is, indeed, a tubular neighborhood of a smooth $(N-1)$-dimensional
 regular surface $\Memb$ with thickness $\eta<<1$ and
 with a finite number of connected components strictly contained in $\Om$.
Redefine $\nOmint=\Memb_\eta$,
$\nOmoutuno=\nOmoutuno_\eta$ and
$\nOmoutdue=\nOmoutdue_\eta$,
so that our domain becomes $\Om=\nOmoutuno_\eta\cup\nOmoutdue_\eta\cup\Memb_\eta$.

Let $\lfint$, $\lfout$, $\alpha$ be as in Subsection \ref{s:thick} and
define $\ndfbothe^\eta (x) = \lfint$ in $\nOmoutuno_\eta$, $\ndfbothe^\eta (x) = \lfout$ in
$\nOmoutdue_\eta$, $\ndfbothe^\eta (x) = 0$ in $\Memb_\eta$, $\nlfbothe^\eta (x) = 0$ in
$\nOmoutuno_\eta\cup\nOmoutdue_\eta$, $\nlfbothe^\eta (x) = {\alpha}/\eta$ in $\Memb_\eta$.
The choice of the scaling $1/\eta$ is designed to let the specific permittivity of the interface to blow up
as $\eta\to 0$. This is essential to allow conduction ``along" the concentrated membrane, as required
by the fact that in the thin interface model we have a Laplace-Beltrami equation on the membrane $\Memb$.

Denoting by $u^\eta$ the solution of problem \eqref{eq1}-\eqref{eq5},
we can rewrite its new weak formulation as
  \begin{multline}\label{eq:a1400}
    \int_{0}^{T}\!\!\int_{\nOmoutuno_\eta\cap\nOmoutdue_\eta} \ndfbothe^\eta \nabla\soluz^\eta\cdot\nabla{\phi}\di x\di t-
    \frac{\alpha}{\eta}\int_{0}^{T}\!\!\int_{\Memb_\eta}  \nabla \soluz^\eta\cdot\nabla\phi_t \di x\di t
    \\
    =
    \int_{0}^{T}\!\!\int_{\Om}f \phi \di x\di t+
    \frac{\alpha}{\eta}\int_{\Memb_\eta} \nabla\overline u_0\cdot\nabla{\phi}(0)\di x\,,
  \end{multline}
  for every test function $\phi\in \CC^\infty(\overline{\Om_T})$ such that $\phi$ has compact support in $\Om$ for every
  $t\in(0,T)$ and $\phi(\cdot,T)=0$ in $\Om$.

  In order to pass to the limit for $\eta\to 0$ in the previous equation, we consider smooth test functions
  $\testeta$ as before such that
$\nabla\testeta \sim \beltramigrad\testeta$ in $\Memb^\eta$
and $\nabla\testeta$ is stable in $\nOmoutuno_\eta\cup\nOmoutdue_\eta$.
Such testing functions can be constructed by a suitable process of interpolation (see \cite{Amar:Andreucci:Gianni:Timofte:2017A,Amar:Gianni:2016C}).
Inserting $\testeta$ in \eqref{eq:a1400}, it formally follows that
\begin{multline}\label{eq:a1401}
    \int_{0}^{T}\!\!\int_{\nOmoutuno_\eta\cap\nOmoutdue_\eta} \ndfbothe^\eta \nabla\soluz^\eta\cdot\nabla{\testeta}\di x\di t-
    \frac{\alpha}{\eta}\,\eta\int_{0}^{T}\!\!\int_{\Memb}  \nabla \soluz^\eta\cdot\beltramigrad\testeta_t \di \sigma\di t
    \\
    \sim
    \int_{0}^{T}\!\!\int_{\Om}f \testeta \di x\di t+
    \frac{\alpha}{\eta}\,\eta\int_{\Memb} \nabla\overline u_0\cdot\beltramigrad{\testeta}(0)\di \sigma\,.
  \end{multline}
Thus, taking into account that $\nabla \soluz^\eta\cdot\beltramigrad\testeta_t=\beltramigrad\soluz^\eta\cdot\beltramigrad\testeta_t$ and
$\nabla\overline u_0\cdot\beltramigrad{\testeta}(0)=\beltramigrad\overline u_0\cdot\beltramigrad{\testeta}(0)$,
and passing to the limit for $\eta\to 0$, we obtain
\begin{multline*}
  \int_{0}^{T}\!\!\int_{\Om} \lfbothe \nabla u\cdot\nabla\phi \di x\di \tau
  -{\alpha}\int_0^T\!\!\int_{\Memb} \beltramigrad u\cdot \beltramigrad \phi_t \di\sigma\di\tau
  \\
  =\int_{0}^{T}\!\!\int_{\Om}  f\phi \di x\di \tau+
  {\alpha}\int_{\Memb} \beltramigrad\overline u_0\cdot\beltramigrad\phi(x,0)\di\sigma \,,
\end{multline*}
which is exactly the weak formulation of the problem with thin membranes (see \eqref{eq:PDEin}--\eqref{eq:InitData}).

\subsection{The problem with thin membranes}
\label{s:thin_prb}

The typical geometrical setting is displayed in Figure~\ref{fig:thin}.
Here we give, for the sake of clarity, its detailed formal definition.

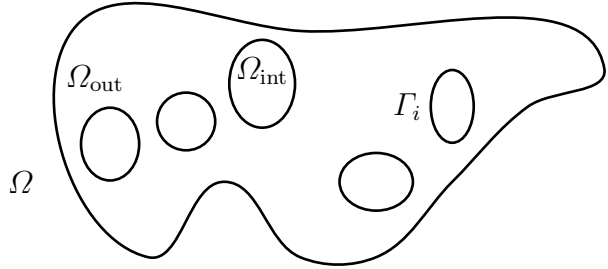
\begin{figure}[htbp]
  \begin{center}
    \begin{pspicture}(12,5)
  \rput(6,2){
    \psccurve[linewidth=1pt](-2,-1)(-1,0)(0,-1)(1,-1)(2,0)(3,1)(4,1.5)(0,2)(-3,2)
    \rput[r](-3.5,0){\mbox{$\Om$}}
      \psellipse[linewidth=1pt](-.5,1.3)(0.45,0.6)
      \psellipse[linewidth=1pt](2,1)(0.3,0.5)
      \psellipse[linewidth=1pt](-1.5,.8)(0.4,0.4)
      \psellipse[linewidth=1pt](-2.5,.5)(0.4,0.5)
      \psellipse[linewidth=1pt](1,0)(.5,.4)%
      \rput[r](1.6,1){\mbox{$\Memb_{i}$}}%
      \rput[b](-.5,1.3){\mbox{$\Omint$}}
      \rput[b](-2.7,1.2){\mbox{$\Omout$}}
  }
\end{pspicture}
 
    \caption{The geometrical setting in the case of thin membranes:  The connected components of $\Memb$ are labelled as $\Memb_{i}$.}
    \label{fig:thin}
  \end{center}
\end{figure}

Let $\Om$ be an open connected bounded subset of $\RN$ such that
$\Om=\Omint\cup\Omout\cup\Memb$, where $\Omint$ and $\Omout$ are
two disjoint open subsets of $\Om$, and
$\Memb=\partial\Omint\cap\Om=\partial\Omout\cap\Om$.
The region $\Omout$ [respectively, $\Omint$] corresponds to the outer
phase [respectively, the inclusions], while
$\Memb$ is the interface. We assume that $\Omout$ is connected (while $\Omint$ could be
connected or not), $\Memb$ is the union of a finite number (say $m\geq 1$)
of connected components $\Memb_i$, and ${\rm dist}(\Memb,\partial\Om)>0$.
We assume also that $\Omout,\Omint$ have regular boundary (that is $\partial\Om$ and $\Memb$
are smooth).
Finally, let $\nu$ denote the normal unit vector to $\Memb$ pointing into $\Omout$.
\medskip

Let us consider the problem
\begin{alignat}2
  \label{eq:PDEin}
  -\Div(\lfbothe \nabla u)&=f\,,&\qquad &\text{in $(\nOmoutuno\cup\nOmoutdue)\times(0,T)$;}
  \\
  \label{eq:FluxCont}
  [u] &=0 \,,&\qquad &\text{on
  $\Memb_T$;}
  \\
  \label{eq:Circuit}
  -\alpha\beltrami u_{ t}
  &=[\lfbothe \nabla u  \cdot \nu]\,,&\qquad
  &\text{on $\Memb_T$;}
  \\
  \label{eq:BoundData}
  u(x,t)&=0\,,&\qquad&\text{on $\partial\Om\times(0,T)$;}
  \\
  \label{eq:InitData}
 \beltramigrad u(x,0)&=\beltramigrad\overline u_0(x)\,,&\qquad&\text{on $\Memb$,}
\end{alignat}
where we denote
\begin{equation}
  \label{eq:jump}
  [u] = u^{\text{out}}  -  u^{\text{int}} \,,
\end{equation}
the same notation being employed also for other quantities.
Here, $\lfbothe=\lfint$ in $\Omint$, $\lfbothe=\lfout$ in $\Omout$ and $\lfint,\lfout,\alpha, f$
are as in the previous section, while $\overline u_0\in H^1(\Memb)$ .

Since problem \eqref{eq:PDEin}--\eqref{eq:InitData} is not standard, in order to define a proper notion of weak
solution, let us set
\begin{equation}
 \label{eq:space2}
 \XX_0(\Om) :=\{u\in H^1_0(\Om): tr\vert_\Memb(u) \in H^1(\Memb)\}\,.
\end{equation}
Notice that $\XX_0(\Om)$ is a Hilbert space endowed with the scalar product given by
the sum of the two natural scalar products in $H^1_0(\Om)$ and $H^1(\Memb)$, respectively
(see, for instance, \cite[end of Subsection 2.3]{Amar:Gianni:2016C}).

\begin{defin}\label{d:weak_sol}
We say that $u\in \spaziosoleps$ is a weak solution of problem \eqref{eq:PDEin}--\eqref{eq:InitData} if
\begin{multline}\label{eq:weak_sol}
  \int_{0}^{T}\!\!\int_{\Om} \lfbothe \nabla u\cdot\nabla\phi \di x\di \tau
  -{\alpha}\int_0^T\!\!\int_{\Memb} \beltramigrad u\cdot \beltramigrad \phi_t \di\sigma\di\tau
  \\
  =\int_{0}^{T}\!\!\int_{\Om}  f\phi \di x\di \tau+
  {\alpha}\int_{\Memb} \beltramigrad\overline u_0\cdot\beltramigrad\phi(x,0)\di\sigma \,,
\end{multline}
for every test function $\phi\in \CC^\infty(\overline{\Om_T})$ such that $\phi$ has compact support in $\Om$ for every
$t\in(0,T)$ and $\phi(\cdot,T)=0$ in $\Om$.
\end{defin}

We also state the following energy inequality, which can be obtained rigorously via a regularization
process in the spirit of Proposition~\ref{r:r6} and formally by testing with $u$ problem \eqref{eq:PDEin}--\eqref{eq:jump}
and integrating in time:
\begin{equation}\label{eq:a30}
  \int_{0}^{T}\!\!\int_{\Om} |\nabla u|^2 \di x\di \tau
  +\sup_{t\in(0,T)}\int_{\Memb} |\beltramigrad u|^2 \di\sigma\leq \gamma
  \left(\Vert f\Vert_{L^2(\Om)}+\Vert \overline u_0\Vert_{H^1(\Memb)}\right)\,,
\end{equation}
where $\gamma$ depends on $\lfint,\lfout,\alpha$ and the Poincar\'{e} constant for $\Om$.
Clearly, uniqueness of solutions follows trivially by the previous energy inequality, hence it remains to prove existence.
To this purpose, we first need the following technical result (for similar problems, see \cite{Cioranescu:Damlamian:Li:2013,Li:1989,
Li:Zheng:Tan:Shen:1998}).
\medskip

\begin{prop}\label{t:a3}
Let $\Memb=\cup_i \Permemb_i$, $i=1,\dots,m$, and assume that $ h_i\in H^1(\Memb_i)$.
Then, there exist $c_i\in\R$, $i=1,\dots,m$, such that the solution to problem
\begin{alignat}2
\label{eq:a10}
\Delta w &=0\,,&\qquad &    \text{in $\Omout$;}
\\
\label{eq:a11}
w &=0\,,&\qquad & \text{on $\partial\Om$;}
\\
\label{eq:a12}
w &= h_i+c_i\,, &\qquad &\text{on $\Memb_i$, $i=1,\dots,m$,}
\end{alignat}
satisfies the further condition
\begin{equation}\label{eq:a13}
\int_{\Memb_i}\pder{w}{\nu}\di\sigma=\ell_i\,,\qquad i=1,\dots,m\,,
\end{equation}
where $\ell_i$ are given numbers.
Moreover, the constants $c_i$ are unique.
\end{prop}

\begin{remark}\label{r:r21}
In Proposition~\ref{t:a3}, the assumption $h_i\in H^1(\Permemb_i)$ may be relaxed to
$h_i\in H^{1/2}(\Permemb_i)$, but in our following application (see Theorem \ref{t:a4}), $h$ shall belong to
$H^1(\Permemb)$, since we will need to consider $\beltramigrad h$.
\end{remark}

\begin{remark}\label{r:r3}
In order to state a periodic version of Proposition \ref{t:a3}, in which $\Om$ is replaced by $Y=(0,1)^N$
and periodic boundary conditions are assigned instead of \eqref{eq:a11},
we need to impose
the natural compatibility condition given by $\sum \ell_i=0$. Notice that,
in this case, uniqueness is ensured up to a global additive constant.
\end{remark}

The regular dependence of the constants $c_{i}$ on the data will play
a role in our proof of the following existence result (see Remark~\ref{r:r4}).

\begin{thm}\label{t:a4}
Let $T>0$, $f\in L^2(\Om_T)$ and $\overline u_0\in  H^1(\Memb)$. Then,
problem \eqref{eq:PDEin}--\eqref{eq:InitData} admits a unique solution $u\in\spaziosoleps$.
Moreover, $\beltramigrad u\in H^1\big(0,T; L^2(\Permemb)\big)$.
\end{thm}

Also in this case only the initial gradient is attained by the solution; see also Remark~\ref{r:r8}.

\begin{remark}\label{r:r7}
Notice that $\beltrami u_t$ in \eqref{eq:Circuit} should be
rewritten in the form $\beltramidiv\big((\beltramigrad u)_t\big)$, as it is done in \eqref{eq:a2nn}
(a similar remark applies to \eqref{eq1}), since $u_t$ is not defined as a
Sobolev derivative.
However, the weak formulations \eqref{eq:weak_sol} and \eqref{eq:a14} still are correct.
\end{remark}

\section{Proof of existence for the problem with thick membranes}
\label{s:exist_thick}

We remark that this problem is non-standard since the principal part of the equation is not coercive. For this reason,
we are led to introduce a coercive perturbation of it and to prove the well-posedness of this $\delta$-perturbed problem.
Then, in order to obtain existence for the original problem, we pass
to the limit for $\delta\to 0$, once we have obtained suitable estimates independent of $\delta$.

\begin{thm}\label{t:a1}
Given $\delta>0$ and $f\in L^2(\Om_T)$,
set $\nlfbothed(x)=\alpha$ in $\nOmint$ and
$\nlfbothed(x)=\delta$ in $\nOmoutuno\cup\nOmoutdue$
and let $\ndfbothe $ and $\widetilde u_0$
be defined as above. Then, for any fixed $T>0$, the problem
\begin{alignat}2
\label{eq:a5}
-\Div \big(\nlfbothed\nabla \soluzdeltat+\ndfbothe \nabla\soluzdelta\big) &=f\,,
&\qquad&\text{in $\Om_T$;}
\\
\label{eq:a5bis}
\nabla\soluzdelta(x,0)&=\nabla\widetilde u_0\,,&\qquad&\text{in $\Om$;}
\\
\label{eq:a6bis}
\soluzdelta(x,t)&=0\,,&\qquad &\text{on $\partial\Omega \times(0,T)$,}
\end{alignat}
admits a unique solution $\soluzdelta\in L^2\big(0,T;H^1_0(\Om)\big)\cap H^1(\Om_T)$.
\end{thm}

\begin{proof}
First, we note that the weak formulation of problem \eqref{eq:a5}--\eqref{eq:a6bis} reads as
\begin{equation}\label{eq:a7}
- \!\! \int_{0}^{T}\!\!\int_{\Om}\!\!  \nlfbothed\nabla\soluzdelta \cdot\nabla\phi_t \di x\di t+\!\!
 \int_{0}^{T}\!\!\int_{\Om}\!\!
 \ndfbothe \nabla\soluzdelta\cdot\nabla{\phi}\di x\di t=\!\!
  \int_{0}^{T}\!\!\int_{\Om} \!\! f\phi\di x\di t +\!
  \int_{\Om} \!\nlfbothed\nabla\widetilde u_0\cdot\nabla{\phi}(0)\di x\,,
\end{equation}
for every test function $\phi\in \CC^\infty(\overline{\Om_T})$ such that $\phi$ has compact
support in $\Om$ for every
$t\in(0,T)$ and $\phi(\cdot,T)=0$ in $\Om$. Moreover, as done in \eqref{energy1} and
\eqref{eq:a30}, we have also the energy estimate
\begin{multline}\label{eq:a7bis}
\sup_{t\in(0,T)}\int_{\nOmint}  |\nabla\soluzdelta|^2\di x+
{\delta}\sup_{t\in(0,T)}\int_{\nOmoutuno\cup\nOmoutdue}  |\nabla\soluzdelta|^2\di x
+ \int_{0}^{T}\!\!\int_{\nOmoutuno\cup\nOmoutdue} |\nabla\soluzdelta|^2\di x\di t
\\
\leq  \gamma\left[\int_{0}^{T}\!\!\int_{\Om} f^2\di x\di t
+\int_{\nOmint}  |\nabla\overline u_0|^2\di x+
{\delta}\int_{\nOmoutuno\cup\nOmoutdue}  |\nabla\widetilde u_0|^2\di x\right]
\le \gamma\,,
\end{multline}
where $\gamma$ depends on $\lfint,\lfout,\alpha,
\Vert \overline u_0\Vert_{H^1(\nOmint)}, \Vert f\Vert_{L^2(\Om)}$,
the Poincar\'{e} constant for $\Om$,
but not on  $\delta$. As a consequence of this energy estimate, we obtain that, if a solution does
exist, then it is unique.

In order to prove existence for problem \eqref{eq:a5}--\eqref{eq:a6bis}, for
any given function $h\in L^2\big(0,T;H^1_0(\Om)\big)$, we consider the auxiliary problem
\begin{alignat}2
\label{eq:a8}
-\Div\big(\nlfbothed\nabla v^{\delta}\big)&=\Div(\ndfbothe \nabla h)+f\,,
& \qquad &  \text{in $\Om$;}
\\
\label{eq:a8bis}
 v^{\delta} &=0\,, &\qquad & \text{on $\partial\Om$.}
\end{alignat}
Since $\nlfbothed$ is a strictly positive $L^\infty(\Om)$-function, problem \eqref{eq:a8}--\eqref{eq:a8bis}
is a standard Dirichlet problem, so that, for a.e. $t\in(0,T)$, it admits a unique solution
$ v^{\delta}\in L^2\big(0,T;H^1_0(\Om)\big)$, satisfying for a.e. $t\in (0,T)$
the following energy estimate:
\begin{equation*}
\int_{\nOmint}|\nabla v^{\delta}|^2\di x+
\delta\int_{\nOmoutuno\cup\nOmoutdue}|\nabla v^{\delta}|^2\di x\le
\gamma (\Vert h\Vert_{H^1_0(\Om)}\Vert \nabla v^{\delta}\Vert^{}_{L^2(\Om)}
+\Vert f\Vert_{L^2(\Om)}\Vert v^{\delta}\Vert^{}_{L^2(\Om)})\,,
\end{equation*}
which gives
\begin{equation*}
\Vert v^{\delta}\Vert^{}_{H^1_0(\Om)}\leq
\gamma_\delta(\Vert h\Vert^{}_{H^1_0(\Om)}+\Vert f\Vert^{}_{L^2(\Om)})\,,
\end{equation*}
where $\gamma_\delta$ depends on $\alpha,\lfint,\lfout,\delta$ and the Poincar\'{e}
constant for $\Om$.
Integrating with respect to time the previous inequality, we obtain
\begin{equation}\label{eq:a9}
\Vert v^{\delta}\Vert^{}_{L^2(0,T;H^1_0(\Om))}\leq
\gamma_\delta(\Vert h\Vert^{}_{L^2(0,T;H^1_0(\Om))}+
\Vert f\Vert^{}_{L^2(\Om_T)})\,.
\end{equation}
Hence, we can define the function
\begin{equation}\label{eq:a6}
\soluzdelta(x,t) = \widetilde u_0(x)+\int_0^t v^{\delta}(x,\tau)\di \tau\,,
\end{equation}
which actually belongs to $L^2\big(0,T;H^1_0(\Om)\big)\cap H^1(\Om_T)$.
Now, let us choose $\overline T$ in such a way that $\gamma_\delta\overline T<1$ and consider
the linear operator $L:L^2\big(0,\overline T;H^1_0(\Om)\big)\to
L^2\big(0,\overline T;H^1_0(\Om)\big)\cap H^1(\Om_T)\subset L^2\big(0,\overline T;H^1_0(\Om)\big)$ defined by $L(h)= \soluzdelta$, where $\soluzdelta$ is given by \eqref{eq:a6}. Clearly, $L$ is a
contraction since
\begin{multline*}
\Vert L(h_1)-L(h_2)\Vert^2_{L^2(0,\overline T;H^1_0(\Om))}
=\Vert \soluzdelta_1-\soluzdelta_2\Vert^2_{L^2(0,\overline T;H^1_0(\Om))}
=\Vert \int_0^t(v_1^\delta-v_2^\delta)\di\tau\Vert^2_{L^2(0,\overline T;H^1_0(\Om))}
\\
\leq \int_0^{\overline T}
\left(t\int_0^t\Vert v_1^\delta-v_2^\delta\Vert^2_{H^1_0(\Om)}\di\tau\right)\di t
\leq \frac{\overline T^2}{2}\int_0^{\overline T}
\Vert v_1^\delta-v_2^\delta\Vert^2_{H^1_0(\Om)}\di\tau
\\
\leq \frac{\overline T^2}{2}\Vert v_1^\delta-v_2^\delta\Vert^2_{L^2(0,\overline T;H^1_0(\Om))}
\leq \gamma^2_\delta \frac{\overline T^2}{2}\Vert h_1-h_2\Vert^2_{L^2(0,\overline T;H^1_0(\Om))}
<\frac{1}{2}\Vert h_1-h_2\Vert^2_{L^2(0,\overline T;H^1_0(\Om))}.
\end{multline*}
Therefore, there exists a unique fixed point $\soluzdelta\in L^2(0,\overline T;H^1_0(\Om))$, given by
\eqref{eq:a6}, where $v^\delta$ satisfies
$$
-\Div\big(\nlfbothed\nabla v^{\delta}\big)=\Div(\ndfbothe \nabla \soluzdelta)+f\,,
\qquad  \text{in $\Om\times(0,\overline T)$,}
$$
which is nothing else than the equation \eqref{eq:a5}, since $v^\delta=\soluzdelta_t$.
Finally, since $\overline T$ depends on $\alpha,\lfint,\lfout, \delta$, but not on the initial
condition $\overline u_0$, we can repeat the previous fixed point argument in the intervals
$(\overline T,2\overline T)$, $( 2\overline T,3\overline T)$ and so on, so that we can recover the whole interval $(0,T)$ by iteration.
Here, we employ the regularity in time of $u^\delta$ to define the trace of $u^\delta(t)$ at all time levels.
Then, the thesis is achieved, once we take into account that conditions \eqref{eq:a5bis}, \eqref{eq:a6bis} are clearly satisfied by the function $\soluzdelta$
thus constructed.
\end{proof}

\begin{proof}[Proof of Theorem~\ref{t:a2}]
For any $\delta>0$, let $\soluzdelta\in L^2\big(0,T;H^1_0(\Om)\big)\cap H^1(\Om_T)$ be the
solution of problem \eqref{eq:a5}--\eqref{eq:a6bis}. By the Poincar\'{e} inequality
and integrating the inequality \eqref{eq:a7bis} with respect to the time $t$ in the
interval $(0,T)$, we obtain
$$
\Vert \soluzdelta\Vert^{}_{L^2(0,T;H^1_0(\Om))}\leq\gamma\,,
$$
where $\gamma$ is independent of $\delta$. Then, passing to a subsequence if needed,
it follows that there exists $u\in L^2\big(0,T;H^1_0(\Om)\big)$ such that
$\soluzdelta\wto u$ for $\delta\to 0$ in $L^2\big(0,T;H^1_0(\Om)\big)$. Then, passing to
the limit in the weak formulation \eqref{eq:a7} and taking into account that $\nlfbothed\to\nlfbothe$
strongly in $L^\infty(\Om)$, we obtain that the limit $\soluz$ satisfies \eqref{eq:a14}; i.e.,
$\soluz$ is a solution of problem \eqref{eq1}--\eqref{eq5}. By uniqueness, it follows that
the whole sequence $(\soluzdelta)$ coverges to $\soluz$ and $\soluz$ is the unique solution
of problem \eqref{eq1}--\eqref{eq5}.
\end{proof}

\section{Proof of existence for the problem with thin interfaces}
\label{s:exist_micro}

The main goal of this section is to prove that problem \eqref{eq:PDEin}--\eqref{eq:InitData} admits a solution
$u\in\spaziosoleps$.

\begin{remark}\label{r:r20}
Since the solution of problem \eqref{eq:a10}--\eqref{eq:a12}, in general, belongs only to
$H^1(\Omout)$, we have to specify the meaning of condition \eqref{eq:a13}.
It is quite a standard result (see, for instance, \cite{baiocchi:capelo}), but we prefer to recall
it here for the reader's convenience.

To this purpose, assume that, given $h\in H^1(\Memb)$, $(h_n)\in\CC^\infty(\Memb)$ is a sequence of smooth
functions such that $h_n\to h$ strongly in $H^1(\Memb)$.
Let us set $H^1_\Memb(\Omout):=\{u\in H^1(\Omout): \text{s.t.  $u=0$ on $\partial\Om$}\}$
and, for $n\in\N$, define the linear functional
$A_n:H^1_\Memb(\Omout)\to\R$ as
\begin{equation}\label{eq:a35}
A_n(\phi)=\int_{\Memb}\pder{u_n}{\nu}\phi\di\sigma\,,
\end{equation}
where, recalling that $\Omout$ is smooth, $u_n\in \CC^\infty(\overline{\Omout})$ is the smooth solution of problem \eqref{eq:a10}--\eqref{eq:a12}
corresponding to the boundary datum $h_n\in \CC^{\infty}(\Memb)$ and $c_i$, $i=1,\dots,m$, as above.
Taking into account that $u_n$ is harmonic, we obtain also
\begin{equation*}
A_n(\phi)=-\int_{\Omout} \nabla u_n\cdot\nabla \phi\di x\,.
\end{equation*}
By standard energy estimate, and taking into account the linearity of the problem, we obtain
\begin{equation}\label{eq:a26tris}
\Vert u_n-u_m\Vert_{H^1(\Omout)}\leq \gamma \Vert h_n-h_m\Vert_{H^1(\Memb)}\,,
\qquad \forall n,m\in\N\,,
\end{equation}
so that $u_n\to u\in H^1_\Memb(\Omout)$ strongly in $H^1(\Omout)$,
where $u$ is still the solution of problem \eqref{eq:a10}--\eqref{eq:a12} corresponding to $h$.

Then, by \eqref{eq:a35}, we obtain that there exists a limit functional denoted by $A:H^1_\Memb(\Omout)\to\R$ and defined by
\begin{equation}\label{eq:a42}
A(\phi)=-\int_{\Omout} \nabla u\cdot\nabla \phi\di x\,.
\end{equation}
In particular, by fixing $i\in \{1,\dots,m \}$ and by taking
$\widehat\phi_i\in H^1_\Memb(\Omout)$ such that $\widehat\phi_i\equiv 1$ on $\Memb_i$
and $\widehat\phi_i=0$ on $\Memb_j$, for $j\not=i$,
we have that
\begin{equation}\label{eq:a43}
\int_{\Memb_i}\pder{u_n}{\nu}\di\sigma=-\int_{\Omout}\nabla u_n\cdot\nabla\widehat\phi_i\di x=A_n(\widehat\phi_i)
\to A(\widehat\phi_i)\,,
\end{equation}
so that we can state that, for a given $H^1$ solution of problem \eqref{eq:a10}--\eqref{eq:a12},
condition \eqref{eq:a13} is understood in the sense that $A(\widehat\phi_i)=0$, for $i=1,\dots,m$,
where $A$ is here the operator associated to $w$.
\end{remark}

\begin{proof}[Proof of Proposition~\ref{t:a3}]
For any $h\in H^1(\Memb_j)$, denote by $u_j[h]\in H^1_\Memb(\Omout)$
the solution of the standard elliptic problem
\begin{alignat}2
\label{eq:a14bis}
\Delta v &=0\,,&\qquad &    \text{in $\Omout$;}
\\
\label{eq:a15}
v &=0\,,&\qquad & \text{on $\partial\Om$;}
\\
\label{eq:a16}
v &= 0\,, &\qquad &\text{on $\Memb_i$, $i=1,\dots,m$, $i\not=j$;}
\\
\label{eq:a16bis}
v &= h\,, &\qquad &\text{on $\Memb_j$.}
\end{alignat}
Then, taking into account the linearity of the problem,
the solution to \eqref{eq:a10}--\eqref{eq:a12} can be written in the form
\begin{equation}\label{eq:a17}
w= \sum_{j=1}^m u_j[h_j]+u_j[c_j]\,.
\end{equation}
Then, the conditions \eqref{eq:a13} become
\begin{equation}\label{eq:a18}
\sum_{j=1}^m \left\{\int_{\Memb_i}\pder{u_j[h_j]}{\nu}\di\sigma+
c_j\int_{\Memb_i}\pder{u_j[1]}{\nu}\di\sigma\right\}=\ell_i\,,
\quad i=1,\dots,m\,.
\end{equation}
Upon defining
\begin{equation}\label{eq:a18bis}
a_{ij}=\int_{\Memb_i}\pder{u_j[1]}{\nu}\di\sigma\qquad\text{and}\qquad
G_i=-\sum_{j=1}^m \int_{\Memb_i}\pder{u_j[h_j]}{\nu}\di\sigma+\ell_i\,,
\end{equation}
we can rewrite \eqref{eq:a18} as
\begin{equation}\label{eq:a23}
\sum_{j=1}^{m} a_{ij}c_j=G_i\,,\qquad i=1,\dots,m\,.
\end{equation}
We claim that the previous linear system has a unique solution $(c_1,\dots,c_m)$.
Indeed, assume, by contradiction, that the corresponding homogeneous system
\begin{equation}\label{eq:a21}
\sum_{j=1}^{m} a_{ij}d_j=0\,,\qquad i=1,\dots,m\,,
\end{equation}
admits a nonzero solution.
It is easily seen that the function $w=\sum_j u_j[d_j]$ solves \eqref{eq:a10}--\eqref{eq:a12}
with $h_j=0$ and $c_j=d_j$, for every $j=1,\dots,m$, and it satisfies also the conditions \eqref{eq:a13}
with $\ell_i=0$, $i=1,\dots,m$, since
\begin{equation}\label{eq:a20}
\int_{\Memb_i}\pder{w}{\nu}\di\sigma=\sum_{j=1}^m d_j\int_{\Memb_i}\pder{u_j[1]}{\nu}\di\sigma
=\sum_{j=1}^{m} a_{ij}d_j=0\,,\qquad i=1,\dots,m\,.
\end{equation}
Now, let $k\in\{1,\dots,m\}$ be an index such that
$$
d_k=\max_{j=1,\dots,m}d_j\,
$$
Then, if $d_k\geq 0$,
by Hopf's Lemma it follows that  $\pder{w}{\nu}<0$ on $\Memb_k$, which contradicts
\eqref{eq:a20} for $i=k$. A similar argument holds when $d_k<0$.
Hence, the linear system \eqref{eq:a21} admits only the null solution, which implies that the matrix
${\mathcal A}:=[a_{ij}]$ has a trivial kernel, that is \eqref{eq:a23} has a unique solution.
\end{proof}

\begin{remark}\label{r:r1}
Let $h\in H^1(\Memb)$ be a given function and consider the problem
\begin{alignat}2
\label{eq:a25}
-\Div (\lfbothe\nabla u) & =0\,,&\qquad &\text{in $\Omint\cup\Omout$;}
\\
\label{eq:a25_quater}
u &=0\,, & \qquad &\text{on $\partial\Om$;}
\\
\label{eq:a25_quinque}
u &=h\,, & \qquad &\text{on $\Memb=\bigcup_{i=1}^m \Permemb_i$.}
\end{alignat}
As a consequence of the previous lemma, where we take $\ell_i=0$, for $i=1,\dots,m$,
it follows that the set
\begin{multline}\label{eq:a24}
\Hzero =\{ h\in H^1(\Memb)\ \text{s.t. the solution $u\in H^1_0(\Om)$ to problem \eqref{eq:a25}
--\eqref{eq:a25_quinque}}
\\
\text{satisfies $\displaystyle \int_{\Memb_i} \lfout\pder{u^{\text{out}}}{\nu}\di\sigma=0$,
$\forall i=1,\dots,m$}\}
\end{multline}
is a non-empty linear space. Moreover, it is not difficult to see that $\Hzero$ is also a closed
subspace of $H^1(\Memb)$. Indeed, if $(h_n)$ is a sequence in $\Hzero$ strongly converging
to $h$ in $H^1(\Memb)$, for $n\to +\infty$,
it follows that the corresponding sequence $(u_n)$ of solutions to
problem \eqref{eq:a25}--\eqref{eq:a25_quinque},
with $h$ replaced by $h_n$, strongly converges in $H^1_0(\Om)$ to the solution $u\in H^1_0(\Om)$ of
problem \eqref{eq:a25}--\eqref{eq:a25_quinque} corresponding
to the limit function $h$.
Moreover, passing to the limit, for $n\to +\infty$, in the weak formulation of \eqref{eq:a25}
for $\Omout$, we obtain
$$
0=-\int_{\Memb_i}\lfout\pder{u_n}{\nu}\di\sigma =
\int_{\Omout}\lfout\nabla u_n\cdot\nabla \varphi\di x\to \int_{\Omout}\lfout\nabla u\cdot\nabla \varphi\di x\,,
$$
for every $\varphi\in H^1_0(\Om)$ supported in a neighbourhood of $\Memb_i$ and $\varphi=1$
on $\Memb_i$. This implies that
$$
0=\int_{\Omout}\lfout\nabla u\cdot\nabla \varphi\di x=-\int_{\Memb_i}\lfout\pder{u}{\nu}\di\sigma\,,
$$
and this can be repeated for every $i=1,\dots,m$, so that $h\in\Hzero$.
Hence, $\Hzero $ is closed, which implies
that it is a Banach space.
\end{remark}

\begin{remark}\label{r:r2}
Let $h_j\in H^1(\Memb_j)$. For $i=1,\dots,m$, if we fix $\widehat\phi_i \in H^1_{\Memb}(\Omout)$
such that $\widehat\phi_i=1$ on $\Memb_i$, $\widehat\phi_i=0$ on  $\Memb_j$ for $j\not=i$, by \eqref{eq:a42} and \eqref{eq:a43}
we have
\begin{multline}\label{eq:a36}
\left\vert \sum_{j=1}^m \int_{\Memb_i}\lfout\pder{u_j[h_j]}{\nu}\di\sigma\right|
=
\left\vert \sum_{j=1}^m \int_{\Omout}\lfout\nabla{u_j[h_j]}\cdot\nabla\widehat\phi_i\di x\right|
\\
\leq\gamma \sum_{j=1}^m\Vert u_j[h_j]\Vert^{}_{H^1(\Omout)}\,\Vert\widehat\phi_i\Vert^{}_{H^1(\Omout)}\leq
\gamma\Vert h\Vert^{}_{H^1(\Memb)}\,,
\end{multline}
where we set $u[h]=\sum_j u_j[h_j]$ and the last inequality is due to the standard energy inequality
for problem \eqref{eq:a10}--\eqref{eq:a12} and the fact that the test function $\widehat\phi_i$ is fixed.
Given $h,g\in  H^1(\Memb)$ with $h\big\vert_{\Memb_j}= h_j$, $g\big\vert_{\Memb_j}= g_j$ and
$h_j,g_j\in H^1(\Memb_j)$, set $G^h_i,G^g_i$ the corresponding numbers defined in \eqref{eq:a18bis}.
Then, by the linearity of problem \eqref{eq:a14bis}--\eqref{eq:a16bis} and
by \eqref{eq:a36}, it follows that
\begin{equation}\label{eq:a40}
|G_i^h-G_i^g|\leq \gamma\Vert h-g\Vert^{}_{H^1(\Memb)}\,.
\end{equation}
Hence, by the first equality in \eqref{eq:a18bis} and by \eqref{eq:a23}, we have
\begin{equation}\label{eq:a41}
|c_j^h-c_j^g|\leq \gamma \Vert h-g\Vert^{}_{H^1(\Memb)}\,,
\end{equation}
where we have taken into account that the constant matrix ${\mathcal A}:=[a_{ij}]$ by Proposition \ref{t:a3} is invertible
and depends only on the geometry.
Here, we have employed also for $c_j^h$ and $c_j^g$ the same notation used for $G^h_i$ and $G^g_i$.
\end{remark}

\begin{remark}\label{r:r4}
We note that, if in \eqref{eq:a16bis} we take $h_j=h_j(t)$
with $h_j\in L^2\big(0,T;H^1(\Memb)\big)$, by standard
regularity results for elliptic equations, it follows that the solution $u_j[h_j]$ of problem
\eqref{eq:a14bis}--\eqref{eq:a16bis}
belongs to the space $L^2\big(0,T;H^1_\Memb(\Omout)\big)$. Then, if also $\ell_i\in L^2(0,T)$,
by \eqref{eq:a18bis} it follows that $G_i\in L^2(0,T)$ and,
by \eqref{eq:a23}, the same holds also for $c_j$.
Analogously, if $h_j\in H^1\big(0,T;H^1(\Memb)\big)$ and $\ell_i\in H^1(0,T)$, then $u_j[h_j]\in H^1\big(0,T;H^1_\Memb(\Omout)\big)$ and again
$G_i\in H^1(0,T)$ and the same holds also for $c_j$.
\end{remark}

{\it Proof of Theorem \ref{t:a4}.}
For a.e. $t\in(0,T)$, let
$\overline u(t)\in H^1_0(\Om)$ be the unique solution of the standard
Dirichlet problem
\begin{alignat}2
\label{eq:a1n}
-\Div(\lfboth\nabla \overline u(t))=& f(t)\,,&\quad &\text{in $\Omint\cup\Omout$;}
\\
\label{eq:a2n} \overline u(t) =& 0\,,&\quad &\text{on $\Memb$.}
\end{alignat}
Clearly, $\overline u$ satisfies
\begin{equation}\label{eq:a3n}
\Vert \overline u(t)\Vert^{}_{H^1_0(\Om)}\leq \const\Vert f(t)\Vert^{}_{L^2(\Om)}\,.
\end{equation}
Moreover, set
\begin{equation}\label{eq:a4n}
\ell_j(t)=-\int_{\partial\Om}\lfout\pder{\overline u(t)}{n}\di\sigma
-\int_{\Omint^j\cup\Omout} f(t)\di x+
\sum_{i\not=j}\int_{\Memb_i}\lfout\pder{\overline u(t)}{\nu}\di\sigma\,,
\end{equation}
where $n$ denotes the outer normal vector to $\partial \Om$.

For a.e. $t\in (0,T)$, let us define
\begin{multline}\label{eq:a24_l}
\Helle =\{ h\in H^1(\Memb)\ \text{s.t. the solution $u\in H^1_0(\Om)$ to problem \eqref{eq:a25}
--\eqref{eq:a25_quinque}}
\\
\text{satisfies $\displaystyle \int_{\Memb_i} \lfout\pder{u^{\text{out}}}{\nu}\di\sigma=\ell_i(t)$,
$\forall i=1,\dots,m$}\}
\end{multline}
where $\ell_i(t)$, $i=1,\dots,m$, are defined in \eqref{eq:a4n}.
Following a similar argument as in Remark \ref{r:r1}, one can easily prove
that $\Helle$ endowed with the distance defined by
$$
\delle(h_1,h_2) = \Vert h_1-h_2\Vert_{H^1(\Memb)}\,,\qquad \forall h_1,h_2\in\Helle\,,
$$
is a complete metric space.
Let $h\in L^2(0,T;\Helle)$ and, for a.e. $t\in(0,T)$, let $\widetilde u(t)$
be the solution of problem \eqref{eq:a25}--\eqref{eq:a25_quinque}
with $h$ replaced by $h(t)$.
Clearly, the unique solution $\widetilde u(t)\in H^1_0(\Om)$ satisfies
\begin{equation}\label{eq:a2}
\Vert \widetilde u(t)\Vert^{}_{H^1_0(\Om)}\leq \gamma \Vert h(t)\Vert_{H^1(\Memb)}\,,
\end{equation}
where $\Vert h(t)\Vert^{}_{H^1(\Memb)}$ depends obviously on $\ell_1,\dots ,\ell_m$, and hence on
$\Vert f(t)\Vert^{}_{L^2(\Om)}$.
Set $u(t)=\overline u(t)+\widetilde u(t)$, which satisfies
\begin{alignat}2
\label{eq:a25n}
-\Div (\lfbothe\nabla u) & =f(t)\,,&\qquad &\text{in $\Omint\cup\Omout$;}
\\
\label{eq:a25_quatern}
u &=0\,, & \qquad &\text{on $\partial\Om$;}
\\
\label{eq:a25_quinquen}
u &=h(t)\,, & \qquad &\text{on $\Memb$.}
\end{alignat}
Starting from $u(t)$, solve the problem
\begin{alignat}2
\label{eq:a3}
 - \alpha\beltrami v(t) & =\left[\lfbothe\pder{ u(t)}{\nu}\right]\,,&
 \qquad &\text{in $\Memb$;}
\\
\label{eq:a28}
\int_{\Memb_i} v(t)\di \sigma & =0\,.&\qquad & i=1,\dots,m.
\end{alignat}
First, we note that $\left[\lfbothe\pder{u(t)}{\nu}\right]\in H^{-1}(\Memb)$; indeed,
taking into account that $u(t)\in H^1_0(\Om)$ solves \eqref{eq:a25n} and setting, for $w\in H^1(\Permemb)$,
$$
\langle\left[\lfbothe \pder{ u(t)}{\nu}\right],w\rangle
:=-\int_{\Omout} \lfout\nabla u(t)\cdot\nabla w\di x-
\int_{\Omint}\lfint\nabla u (t)\cdot\nabla w\di x\,,
$$
where $\langle\cdot,\cdot\rangle$ denotes the duality pairing between $H^{-1}(\Memb)$ and $H^1(\Memb)$
and $w$ is assumed to be extended from $H^1(\Memb)$ to an $H^1_0(\Om)$-function such that $\Vert w\Vert_{H^1_0(\Om)}\leq\gamma
\Vert w\Vert_{H^1(\Memb)}$,
it follows that $\left[\lfbothe \pder{ u(t)}{\nu}\right]$ is a linear and continuous functional on $H^1(\Memb)$, since
\begin{equation}\label{eq:a26}
|\langle\left[\lfbothe \pder{ u(t)}{\nu}\right],w\rangle|\leq \gamma \Vert  u(t)\Vert_{H^1_0(\Om)}\Vert w\Vert_{H^1_0(\Om)}
\leq \gamma (\Vert h(t)\Vert_{H^1(\Memb)}+\Vert f(t)\Vert_{L^2(\Om)})\Vert w\Vert_{H^1(\Memb)}\,.
\end{equation}
Moreover, since $ u(t)$ satisfies \eqref{eq:a25n} inside $\Omint$, it follows that
$$
\int_{\Memb_i}\lfint\pder{u^{\text{int}}(t)}{\nu}\di\sigma =-\int_{\Omint^i}f(t)\di x\,,
$$
while
\begin{multline*}
\int_{\Memb_i}\lfout\pder{u^{\text{out}}(t)}{\nu}\di\sigma =
\int_{\Omout} f(t)\di x-\sum_{j\not=i}\int_{\Memb_j}\lfout\pder{\overline u(t)}{\nu}\di\sigma
+\int_{\partial\Om}\lfout\pder{\overline u(t)}{n}\di\sigma
+\ell_i(t) =
\\
\int_{\Omout} f(t)\di x-\sum_{j\not=i}\int_{\Memb_j}\lfout\pder{\overline u(t)}{\nu}\di\sigma
+\int_{\partial\Om}\lfout\pder{\overline u(t)}{n}\di\sigma
\\
-\int_{\partial\Om}\lfout\pder{\overline u(t)}{n}\di\sigma
-\int_{\Omint^i\cup\Omout} f(t)\di x+
\sum_{j\not=i}\int_{\Memb_j}\lfout\pder{\overline u(t)}{\nu}\di\sigma
=-\int_{\Omint^i} f(t)\di x
\end{multline*}
by the choice of $h\in\Helle$.
Therefore, the compatibility condition
$$
\int_{\Memb_i}\left[\lfbothe\pder{u(t)}{\nu}\right]\di\sigma =0
$$
is satisfied, so that, from standard results, for a.e. $t\in(0,T)$,
there exists a unique solution $v(t) \in H^1(\Memb)$ of \eqref{eq:a3} such that
\begin{equation}\label{eq:a4}
\Vert v(t)\Vert_{H^1(\Memb)}\leq
\gamma \Vert\left[\lfbothe\pder{u(t)}{\nu}\right]\Vert^{}_{H^{-1}(\Memb)} \leq
\gamma (\Vert h(t)\Vert^{}_{H^1(\Memb)}+\Vert f(t)\Vert^{}_{L^2(\Om)})\,,
\end{equation}
where the last inequality is due to \eqref{eq:a26}.

Define
\begin{equation}\label{eq:a27}
w(x,t)=\overline u_0(x)+\int_0^t v(x,\tau)\di\tau\in H^1\big((0,T);H^1(\Memb)\big)\,,
\end{equation}
and
\begin{equation}\label{eq:a27bis}
\widetilde w(x,t)= w(x,t) + \sum_{i=1}^m \widetilde c_i(t)\chi_{\Memb_i}(x)\,,
\end{equation}
where, for $i=1,\dots,m$, $\widetilde c_i(t)$ are the constants given in Proposition \ref{t:a3} with $h_i(x,t)=w(x,t)$ and $\ell_i$ given by \eqref{eq:a4n}.
Now, set $\overline T =(2\sqrt\gamma)^{-1}$, where $\gamma$ is the constant given in the last line
of \eqref{eq:a31}, and consider the operator $L:L^2\big(0,\overline T;\Helle\big)\to L^2\big(0,\overline T;\Helle\big)$,
defined by $L(h)=\widetilde w$, with $\widetilde w$ given by \eqref{eq:a27bis}. Clearly,
$L$ is a contraction since
\begin{multline}\label{eq:a31}
\int_0^{\overline T} \left[\delle\big(L(h_1),L(h_2)\big)\right]^2\di t
=\Vert \widetilde w_1-\widetilde w_2\Vert^2_{L^2(0,\overline T;H^1(\Memb))}
\\
=\Vert \int_0^t(v_1-v_2)\di\tau+\sum_{i=1}^m \big(\widetilde c^1_i(t)-\widetilde c^2_i(t)\big)\chi_{\Memb_i}\Vert^2_{L^2(0,\overline T;H^1(\Memb))}
\\
\leq \gamma\left\{\int_0^{\overline T}
\left(t\int_0^t\Vert v_1(\tau)-v_2(\tau)\Vert^2_{H^1(\Memb)}\di\tau\right)\di t
+  \sum_{i=1}^m |\Memb_i|\int_0^{\overline T}\big(\widetilde c^1_i(t)-\widetilde c^2_i(t)\big)^2\di t \right\}
\\
\leq \gamma\left\{\int_0^{\overline T}
\left(t\int_0^t\Vert v_1(\tau)-v_2(\tau)\Vert^2_{H^1(\Memb)}\di\tau\right)\di t+ \sum_{i=1}^m |\Memb_i|\int_0^{\overline T}
\Vert w_1(t)-w_2(t)\Vert^2_{H^1(\Memb_i)}\di t \right\}
\\
\leq \gamma\left\{\int_0^{\overline T}
\left(t\int_0^t\Vert v_1(\tau)-v_2(\tau)\Vert^2_{H^1(\Memb)}\di\tau\right)\di t+ \int_0^{\overline T}\!\!\left(t
\int_0^t\Vert v_1(\tau) -v_2(\tau) \Vert^2_{H^1(\Memb)}\di\tau\!\right)\!\!\!\di t\!\right\}
\\
\leq \gamma\overline T^2\int_0^{\overline T}
\Vert v_1(\tau)-v_2(\tau)\Vert^2_{H^1(\Memb)}\di t
\leq \gamma\overline T^2\Vert h_1-h_2\Vert^2_{L^2(0,\overline T;H^1(\Memb))}
\\
=\frac{1}{2}\Vert h_1-h_2\Vert^2_{L^2(0,\overline T;H^1(\Memb))},
\end{multline}
where we reason as in \eqref{eq:a41} (with $h$ and $g$ replaced by $w_1$ and $w_2$ respectively).
We also use an obvious version of \eqref{eq:a4} written for $f=0$, which readily follows from
the definition of $v$.
Therefore, there exists a unique fixed point $\widetilde w\in L^2\big(0,\overline T;\Helle\big)$.
Consider the function $u$ defined as in \eqref{eq:a25n}--\eqref{eq:a25_quinquen} with $h=\widetilde w$ being
the fixed point. Then, according to our definition,
\begin{equation}\label{eq:a1nn}
\beltramigrad u(x,t)=\beltramigrad \widetilde w(x,t)=\beltramigrad w(x,t)= \beltramigrad\overline u_0(x)
+\int_0^t \beltramigrad v(x,\tau)\di \tau\,,
\end{equation}
where $v$ is the solution of \eqref{eq:a3}--\eqref{eq:a28} for the just defined $u$. Then, clearly,
\begin{equation}\label{eq:a2nn}
-\alpha \beltramidiv(\beltramigrad u)_t=\left[\lfboth\nabla u\cdot \nu\right]\,.
\end{equation}
Finally, $\beltramigrad u(x,0)=\beltramigrad \overline u_0(x)$; then, $u$ solves our problem \eqref{eq:PDEin}--\eqref{eq:InitData}
in $(0,\overline T)$. Note that actually $\beltramigrad u$ is continuous in time (in the $L^2$-norm) either
owing to an analogue of Proposition \ref{r:r6} or simply by \eqref{eq:a1nn} above, so that we may choose
$\beltramigrad u(x,\overline T)$ as a new initial data.
In this fashion we cover the interval $(0,T)$ with a finite number of steps of width $\overline T$, which does not depend on the
initial data.
\hfill$\square$

\begin{coroll}\label{c:c1}
Let $T>0$ and $\overline u_0\in H^1(\Memb)$. If $f\in H^1\big(0,T;L^2(\Om)\big)$,
then the solution $u$ of problem \eqref{eq:PDEin}--\eqref{eq:InitData}
belongs to $H^1\big(0,T;\XX_0(\Om)\big)$.
\end{coroll}

\begin{proof}
It is a direct consequence of the construction in \eqref{eq:a27bis}, when we take into account that
$\ell_j\in H^1(0,T)$, $j=1,\dots, m$, as it follows from \eqref{eq:a4n}, and we recall Remark \ref{r:r4}.
\end{proof}

\begin{remark}\label{r:r8}
Clearly, in the case $f\in H^1\big(0,T;L^2(\Om)\big)$, the previous corollary implies that $u(x,0)$ is defined a.e. in
$\Om$; however, in general, it does not coincide with $\overline u_0$, but it is
``rearranged" by adding in each connected component of $\Permemb$ the suitable constant provided by Proposition \ref{t:a3},
as explained in Remark \ref{r:r6bis}.
\end{remark}

\begin{remark}\label{r:r5}
We note that, if $\Om=Y=(0,1)^N$ and we replace condition \eqref{eq:BoundData} with the requirement
that the solution is $Y$-periodic with respect to the spatial variable and has null mean average on $Y$
for a.e. $t\in (0,T)$, then the corresponding problem \eqref{eq:PDEin}--\eqref{eq:Circuit},
\eqref{eq:InitData} still admits a unique solution $u\in L^2\big(0,T;H^1_\#(Y)\cap H^1(\Memb)\big)$,
when also the initial datum $\overline u_0$ and the source $f$ are assumed to be $Y$-periodic and
$f$ has null mean average on $Y$.
\end{remark}

\end{document}